\documentclass[11pt]{article}
\usepackage{amsfonts,amscd,amssymb, amsmath}
\usepackage[usenames]{color}

\usepackage{graphicx}

\usepackage{euscript}
\usepackage[all,tips]{xy}
\usepackage{epic,eepic}

\usepackage{color}
\usepackage{tikz}

\newtheorem{definition}{Definition}[section]

\newtheorem{proposition}[definition]{Proposition}
\newtheorem{corollary}[definition]{Corollary}
\newtheorem{remark}[definition]{Remark}

\newtheorem{theorem}[definition]{Theorem}

\def\rawo\lonra{\longrightarrow}

\def\ot{\otimes}

\allowdisplaybreaks[4]
\newenvironment{proof}{{\it Proof.}}{\hfill $ \square $ \vskip 4mm}

\begin{document}

\title{$\{\sigma , \tau \}$-Rota-Baxter operators, infinitesimal Hom-bialgebras and the associative 
(Bi)Hom-Yang-Baxter equation}
\author{Ling Liu \\
College of Mathematics, Physics and Information Engineering,\\
Zhejiang Normal University, \\
Jinhua 321004, China \\
e-mail: ntliulin@zjnu.cn \and Abdenacer Makhlouf \\
Universit\'{e} de Haute Alsace, \\
IRIMAS - d\'epartement de  Math\'{e}matiques, \\
6, rue des fr\`{e}res Lumi\`{e}re, F-68093 Mulhouse, France\\
e-mail: Abdenacer.Makhlouf@uha.fr \and Claudia Menini \\
University of Ferrara, 
Department of Mathematics, \\
Via Machiavelli 35, Ferrara, I-44121, Italy \\
e-mail: men@unife.it \and Florin Panaite \\
Institute of Mathematics of the Romanian Academy,\\
PO-Box 1-764, RO-014700 Bucharest, Romania\\
e-mail: florin.panaite@imar.ro }
\maketitle

\begin{abstract}
We introduce the concept of 
$\{\sigma , \tau \}$-Rota-Baxter operator, as a twisted version of a Rota-Baxter operator of weight zero.  We show how to 
obtain a certain $\{\sigma , \tau \}$-Rota-Baxter operator from a solution of the associative (Bi)Hom-Yang-Baxter equation, and, 
in a compatible way, 
a Hom-pre-Lie algebra from an infinitesimal Hom-bialgebra. \\

\begin{small}
\end{small}
\end{abstract}

\section{Introduction}

Hom-type algebras appeared in the Physics literature related to quantum deformations of algebras of vector fields; 
these types of algebras satisfy a modified version of the Jacobi identity involving a homomorphism, and were called 
Hom-Lie algebras by Hartwig, Larsson and Silvestrov in \cite{hls}, \cite{larsson}. Afterwards, Hom-analogues of 
various classical algebraic structures have been introduced in the literature, such as Hom-(co)associative (co)algebras, 
Hom-dendriform algebras, Hom-pre-Lie algebras etc. Recently, structures of a more general type have been 
introduced in \cite{gmmp}, called BiHom-type algebras, for which a classical algebraic identity is twisted by 
two commuting homomorphisms (called structure maps). 

Infinitesimal bialgebras were introduced by Joni and Rota in \cite{jonirota} (under the name infinitesimal coalgebra). 
The current name is due to Aguiar, who developed a theory for them in a series of papers (\cite{aguiarcontemp,aguiarjalgebra, aguiarlectnotes}). It turns out that infinitesimal bialgebras have connections with some other 
concepts such as Rota-Baxter operators, pre-Lie algebras, Lie bialgebras etc. Aguiar discovered a large class of 
examples of infinitesimal bialgebras, namely he showed that the path algebra of an arbitrary quiver carries a natural structure of 
infinitesimal bialgebra. In an analytical context, infinitesimal bialgebras have been used in \cite{voiculescu} by Voiculescu 
in free probability theory. 

The Hom-analogue of infinitesimal bialgebras, called 
infinitesimal Hom-bialgebras, was introduced and studied by Yau in \cite{yauinf}. He extended to the Hom-context some 
of Aguiar's results; however, there exist several basic results of Aguiar that do not have a Hom-analogue 
in Yau's paper. It is our aim here to complete the study, by proving those Hom-analogues.  

The associative Yang-Baxter equation was introduced by Aguiar in \cite{aguiarcontemp}. 
Let $(A, \mu )$ be an associative algebra and $r=\sum _ix_i\otimes y_i\in A\otimes A$; then $r$ is called a solution of the  
associative Yang-Baxter equation if 
\begin{eqnarray*}
&&\sum _{i, j}x_i\otimes y_ix_j\otimes y_j=\sum _{i, j}x_ix_j\otimes y_j\otimes y_i+
\sum _{i, j}x_i\otimes x_j\otimes y_jy_i. 
\end{eqnarray*}
In this situation, Aguiar noticed in \cite{aguiarlmp} that the map $R:A\rightarrow A$, $R(a)=\sum _ix_iay_i$, is a 
Rota-Baxter operator of weight zero. We recall (see for instance \cite{Guo}) that 
if $B$ is an algebra and 
$R:B\rightarrow B$ is a linear map, then $R$ is called a Rota-Baxter operator of weight zero if 
\begin{eqnarray*}
&&R(a)R(b)=R(R(a)b+aR(b)), \;\;\;\forall \;\;a, b\in B. 
\end{eqnarray*}
Rota-Baxter operators appeared first in the work of Baxter in probability and the study of fluctuation theory, and were 
intensively studied by Rota in connection with combinatorics. Rota-Baxter operators occured also in other 
areas of mathematics and physics,  notably in the seminal work of Connes and Kreimer \cite{ConnesKreimer} 
concerning a Hopf algebraic approach to renormalization in quantum field theory. 

The Hom-analogue of the associative Yang-Baxter equation was introduced by Yau in \cite{yauinf}, but without exploring 
the relation between this new equation and Rota-Baxter operators. Our first aim is to obtain Hom and BiHom-analogues of Aguiar's observation mentioned above, expressing a relationship between  Hom and BiHom-analogues of the associative 
Yang-Baxter equation and certain generalized Rota-Baxter operators.
The 
BiHom-analogue of the associative Yang-Baxter equation is defined
as follows. Let $(A, \mu , \alpha , \beta )$ be a BiHom-associative algebra and $r=\sum _ix_i\otimes y_i\in A\otimes A$ such that 
$(\alpha \otimes \alpha )(r)=r=(\beta \otimes \beta )(r)$; we say that $r$ is a solution of the associative BiHom-Yang-Baxter equation if 
\begin{eqnarray*}
&&\sum _{i, j}\alpha (x_i)\otimes y_ix_j\otimes \beta (y_j)=\sum _{i, j}x_ix_j\otimes \beta (y_j)\otimes \beta (y_i)+
\sum _{i, j}\alpha (x_i)\otimes \alpha (x_j)\otimes y_jy_i. 
\end{eqnarray*}
To such an element $r$ we want to associate a certain linear map $R:A\rightarrow A$, that will turn out to be a 
twisted version of a 
Rota-Baxter operator of weight zero. More precisely, the map $R$ is defined by 
\begin{eqnarray*}
&&R:A\rightarrow A, \;\;\;R(a)=\sum _i\alpha \beta ^3(x_i)(a\alpha ^3(y_i))=
\sum _i(\beta ^3(x_i)a)\alpha ^3\beta (y_i), \;\;\;\;\;\forall \;\;a\in A,  
\end{eqnarray*}
which in the Hom case (i.e. for $\alpha =\beta $) reduces to $R(a)=\sum _i\alpha (x_i)(ay_i)=\sum _i (x_ia)\alpha (y_i)$, 
for all $a\in A$, 
and the equation it satisfies is (see Theorem \ref{ABRB}) 
\begin{eqnarray*}
&&R(\alpha \beta (a))R(\alpha \beta (b))=R(\alpha \beta (a)R(b)+R(a)\alpha \beta (b)), \;\;\;\forall \;\;a, b\in A. 
\end{eqnarray*}
We call a linear map satisfying this equation an $\alpha \beta $-Rota-Baxter operator (of weight zero). This is a 
particular case of the following concept we introduce and study in this paper. Let $B$ be an algebra, $\sigma , \tau :B\rightarrow B$ algebra maps and 
$R:B\rightarrow B$ a linear map. We call $R$ a $\{\sigma , \tau\}$-Rota-Baxter operator if 
\begin{eqnarray*}
&&R(\sigma (a))R(\tau (b))=R(\sigma (a)R(b)+R(a)\tau (b)), \;\;\;\forall \;\;a, b\in B. 
\end{eqnarray*}
This concept is a sort of modification of the concept of $(\sigma , \tau )$-Rota-Baxter operator 
introduced in \cite{panvan} (inspired by an example in \cite{manchon}). In Section \ref{sec2} we 
prove that certain classes of $\{\sigma , \tau\}$-Rota-Baxter operators have 
similar properties to those of a usual Rota-Baxter operator of weight zero (see Theorem \ref{simprop} and 
its corollaries, and Proposition \ref{analogLie}). 

Our second aim is to extend  to infinitesimal Hom-bialgebras the following result from \cite{aguiarlectnotes}
providing a left pre-Lie algebra from a given  infinitesimal bialgebra.
\begin{theorem} [Aguiar] \label{zerounu}
Let $(A, \mu , \Delta )$ be an infinitesimal bialgebra, with notation $\mu (a\otimes b)=ab$ and 
$\Delta (a)=a_1\otimes a_2$, for all $a, b\in A$. If we define a new operation on $A$ by $a\bullet b=b_1ab_2$, then 
$(A, \bullet $) is a left pre-Lie algebra. 
\end{theorem}

Let $(A, \mu , \Delta , \alpha )$ be an infinitesimal 
Hom-bialgebra, with notation $\mu (a\otimes b)=ab$ and 
$\Delta (a)=a_1\otimes a_2$, for all $a, b\in A$. We want to define a new multiplication $\bullet $ on $A$, turning it into a 
left Hom-pre-Lie algebra. It is not clear what the formula for this multiplication should be (note for instance that 
the obvious choice $a\bullet b=\alpha (b_1)(ab_2)=(b_1a)\alpha (b_2)$ does not work), and we need to guess it.  
We proceed as follows. 
Recall first the following old result: 
\begin{theorem}[Gel'fand-Dorfman]
Let $(A, \mu )$ be an associative and commutative algebra, with notation $\mu (a\otimes b)=ab$, and $D:A\rightarrow A$ 
a derivation. Define a new multiplication on $A$ by $a\star b=aD(b)$. Then $(A, \star )$ is a left pre-Lie 
algebra (it is actually even a Novikov algebra).
\end{theorem}

We make the following observation: if the infinitesimal bialgebra in Aguiar's Theorem is commutative, then his theorem is a 
particular case of the theorem of Gel'fand and Dorfman. Indeed, by using commutativity, 
the multiplication $\bullet $ becomes 
$a\bullet b=b_1ab_2=ab_1b_2=aD(b)$, 
where we denoted by $D$ the linear map $D:A\rightarrow A$, $D(b)=b_1b_2$, i.e. $D=\mu \circ \Delta $, 
and it is well-known 
(see \cite{aguiarcontemp}) that $D$ is a derivation.  

We want to exploit this observation in order to guess the formula for the multiplication in the Hom case. There, we 
already have an analogue of the Gel'fand-Dorfman Theorem, due to Yau (see \cite{yaunovikov}), saying that if 
$(A, \mu , \alpha )$ is a 
commutative Hom-associative algebra and $D:A\rightarrow A$ is a derivation (in the usual sense) commuting with $\alpha $ 
and we define a new multiplication on $A$ by $a*b=aD(b)$, then $(A, *, \alpha )$ is a left Hom-pre-Lie algebra 
(it is actually even Hom-Novikov).   So, we begin with a commutative infinitesimal Hom-bialgebra 
$(A, \mu , \Delta , \alpha )$ and we define the map $D:A\rightarrow A$ also by the formula $D=\mu \circ \Delta $. 
The problem is that, because of the condition from the definition of an infinitesimal Hom-bialgebra satisfied by $\Delta $ (which involves the map $\alpha $), $D$ is \emph{not} a derivation (so we cannot use Yau's result mentioned above). Instead, 
it turns out that $D$ is a so-called $\alpha ^2$-derivation, that is it satisfies  $D(ab)=\alpha ^2(a)D(b)+D(a)\alpha ^2(b)$. 
So what we need first is a generalization of Yau's version of the Gel'fand-Dorfman Theorem, one that would apply not only 
to derivations but also to $\alpha ^2$-derivations. A generalization dealing  
with $\alpha ^k$-derivations, for $k$ an arbitrary natural number,  is achieved in Proposition \ref{genGD}. 
The outcome is a left Hom-pre-Lie algebra (actually, a Hom-Novikov algebra) 
whose structure map is $\alpha ^{k+1}$. Coming back to the case $k=2$, 
by applying this result we obtain that, for the commutative infinitesimal Hom-bialgebra we started with, 
we are able to obtain a left Hom-pre-Lie algebra structure on it, with structure map $\alpha ^3$ and 
multiplication $x\bullet y=\alpha ^2(x)D(y)=\alpha ^2(x)(y_1y_2)$, which, by using commutativity and 
Hom-associativity, may be written as $x\bullet y=\alpha (y_1)(\alpha (x)y_2)$. 

We can consider this formula even if the infinitesimal Hom-bialgebra is \emph{not} commutative, and it turns out 
that this is the formula we were trying to guess (see Proposition \ref{infprelie}). 

Let $(A, \mu , \Delta _r)$ be a quasitriangular infinitesimal bialgebra, i.e. the comultiplication is given by the principal 
derivation corresponding to a solution $r=\sum _ix_i\otimes y_i$ of the associative Yang-Baxter equation. 
There are two left pre-Lie 
algebras associated to $A$: the first one is obtained by Theorem \ref{zerounu}, the second is obtained from the 
fact that the Rota-Baxter operator $R:A\rightarrow A$, $R(a)=\sum _ix_iay_i$ provides a dendriform algebra, which 
in turn provides a left pre-Lie algebra. Aguiar proved in \cite{aguiarlectnotes} that these two left pre-Lie algebras coincide. 
Our last result shows that the Hom-analogue of this fact is also true. 

In a subsequent paper (\cite{lmmp3}) we will introduce the BiHom-analogue of infinitesimal bialgebras and prove the 
BiHom-analogue of Theorem \ref{zerounu}. It turns out that things are more complicated than in the Hom case, 
and moreover the result in the Hom case is \emph{not} a particular case of the corresponding result in the BiHom case. 
This comes essentially from the following phenomenon. A BiHom-associative algebra $(A, \mu , \alpha , \beta )$ 
for which $\alpha =\beta $ is the same thing as the Hom-associative algebra $(A, \mu , \alpha )$. But a left 
BiHom-pre-Lie algebra $(A, \mu , \alpha , \beta )$ (as defined in \cite{lmmp2}) for which $\alpha =\beta $ is 
\emph{not} the same thing as the left Hom-pre-Lie algebra $(A, \mu , \alpha )$, unless $\alpha $ is bijective. 

\section{Preliminaries}\label{sec1} 

\setcounter{equation}{0} 

We work over a base field $\Bbbk $. All
algebras, linear spaces etc. will be over $\Bbbk $; unadorned $\otimes $
means $\otimes_{\Bbbk}$. Unless otherwise specified, the
(co)algebras ((co)associative or not) that will appear in what follows are
\emph{not} supposed to be (co)unital, a multiplication $\mu :V\otimes
V\rightarrow V$ on a linear space $V$ is denoted by $\mu
(v\otimes v^{\prime })=vv^{\prime }$, and for a comultiplication $\Delta :C\rightarrow
C\otimes C$ on a linear space $C$ we use a Sweedler-type notation $\Delta
(c)=c_1\otimes c_2$, for $c\in C$. 
For the composition of two maps $f$
and $g$, we will write either $g\circ f$ or simply $gf$. For the identity
map on a linear space $V$ we will use the notation $id_V$. 
\begin{definition} (\cite{gmmp}) 
A BiHom-associative algebra is a 4-tuple $\left( A,\mu ,\alpha ,\beta \right) $, where $A$ is
a linear space, $\alpha , \beta :A\rightarrow A$ 
and $\mu :A\otimes A\rightarrow A$ are linear maps, such that 
$\alpha \circ \beta =\beta \circ \alpha $, 
$\alpha (xy) =\alpha (x)\alpha (y)$, $\beta (xy)=\beta (x)\beta (y)$ and the so-called BiHom-associativity condition
\begin{eqnarray}
\alpha (x)(yz)=(xy)\beta (z) \label{BHassoc}
\end{eqnarray}
hold, for all $x, y, z\in A$. The maps $\alpha $ and $\beta $ (in this order) are called the structure maps
of $A$. 
\end{definition}

A Hom-associative algebra, as defined in \cite{ms1}, is a 
BiHom-associative algebra $\left( A,\mu ,\alpha ,\beta \right) $ for which $\alpha =\beta $. The defining relation, 
\begin{eqnarray}
\alpha (x)(yz)=(xy)\alpha (z), \;\;\;\forall \;\;x, y, z\in A, 
\label{Hass}
\end{eqnarray}
is called the Hom-associativity condition and the map $\alpha $ is called the structure map. 

If $(A, \mu )$ is an associative algebra and $\alpha , \beta :A\rightarrow A$ are two commuting algebra maps, then 
$A_{(\alpha , \beta )}:=(A, \mu \circ (\alpha \otimes \beta ), \alpha , \beta )$ is a BiHom-associative algebra, 
called the Yau twist of $A$ via the maps $\alpha $ and $\beta $. 
\begin{definition} (\cite{ms2}) 
A Hom-coassociative coalgebra is a triple $(C, \Delta, \alpha )$, 
in which $C$ is a linear space, $\alpha  :C\rightarrow C$
and $\Delta :C\rightarrow C\otimes C$ are linear maps, such that 
$(\alpha \otimes \alpha )\circ \Delta = \Delta \circ \alpha $ and 
\begin{eqnarray}
&&(\Delta \otimes \alpha )\circ \Delta = (\alpha \otimes \Delta )\circ \Delta . \label{Hcoass}
\end{eqnarray}
The map $\alpha $ is called the structure map  and (\ref{Hcoass}) is called the Hom-coassociativity condition.
\end{definition}

For a Hom-coassociative coalgebra $(C, \Delta , \alpha )$, we will use the extra notation $(id \otimes \Delta )(\Delta (c))=c_1\otimes c_{(2, 1)}
\otimes c_{(2, 2)}$ and $(\Delta \otimes id)(\Delta (c))=c_{(1, 1)}\otimes c_{(1, 2)}\otimes c_2$, for all $c\in C$. 
\begin{definition}
A left pre-Lie algebra is a pair $(A, \mu )$, where $A$ is a a linear space and 
$\mu :A\ot A\rightarrow A$ is a linear map satisfying the condition 
\begin{eqnarray*}
&&x(yz)-(xy)z=y(xz)-(yx)z, \;\;\;  \forall \;\;x, y, z\in A. 
\end{eqnarray*}

A morphism of left pre-Lie algebras from $(A, \mu )$ to $(A', \mu ')$ is a linear map 
$\alpha :A\rightarrow A'$ satisfying $\alpha (xy)=\alpha (x)\alpha (y)$, for all $x, y\in A$. 
\end{definition}
\begin{definition} (\cite{ms1}, \cite{yaunovikov}) 
A left Hom-pre-Lie algebra is a triple $(A, \mu ,
\alpha )$, where $A$ is a linear space and $\mu :A\otimes
A\rightarrow A$ and $\alpha :A\rightarrow A$ are linear maps
satifying $\alpha (xy)=\alpha (x)\alpha (y)$ 
and
\begin{eqnarray}
&&\alpha (x)(yz)-(xy)\alpha (z)=\alpha (y)(xz)-(yx)\alpha (z), 
\end{eqnarray}
for all $x, y, z\in A$. We call $\alpha $ the
structure map of $A$. If moreover the condition  
\begin{eqnarray}
&&(xy)\alpha (z)=(xz)\alpha (y), \;\;\; \forall \;\;x, y, z\in A,  \label{homnovikov}
\end{eqnarray}
is satisfied, then $(A, \mu , \alpha )$ is called a Hom-Novikov algebra. 
\end{definition}

If $(A, \mu )$ is a left pre-Lie algebra and $\alpha :A\rightarrow A$ is a morphism of left pre-Lie algebras, then 
$A_{\alpha }:=(A, \alpha \circ \mu , \alpha )$ is a left Hom-pre-Lie algebra, called the Yau twist of $A$ via the map $\alpha $. 
\begin{definition} (\cite{aguiarcontemp}) 
An infinitesimal bialgebra is a triple $(A, \mu , \Delta )$, in which 
$(A, \mu )$ is an associative algebra, $(A, \Delta )$ is a coassociative coalgebra and $\Delta :A\rightarrow A\otimes A$
is a derivation, that is 
$\Delta (ab)=ab_1\otimes b_2+a_1\otimes a_2b$, for all $a, b\in A$.

A morphism of infinitesimal bialgebras from $(A, \mu , \Delta )$ to $(A', \mu ', \Delta ')$ is a linear map 
$\alpha :A\rightarrow A'$ that is a morphism of algebras and a morphism of coalgebras. 
\end{definition}
\begin{definition}(\cite{yauinf}) An infinitesimal Hom-bialgebra is a 4-tuple $(A, \mu , \Delta, \alpha )$, in which 
$(A, \mu , \alpha )$ is a Hom-associative algebra, $(A, \Delta , \alpha )$ is a Hom-coassociative coalgebra and 
\begin{eqnarray}
&&\Delta (ab)=\alpha (a)b_1\otimes \alpha (b_2)+\alpha (a_1)\otimes a_2\alpha (b), \;\;\;\forall \;\;a, b\in A.
\label{hominf}
\end{eqnarray} 
\end{definition}
\begin{definition} (\cite{yauinf})
Let $(A, \mu , \alpha )$ be a Hom-associative algebra and $r=\sum _ix_i\otimes y_i\in A\otimes A$ such that 
$(\alpha \otimes \alpha )(r)=r$. Define the following elements in $A\otimes A\otimes A$: 
\begin{eqnarray*}
&&r_{12}r_{23}=\sum _{i, j}\alpha (x_i)\otimes y_ix_j\otimes \alpha (y_j), \;\;\;\;\;
r_{13}r_{12}=\sum _{i, j}x_ix_j\otimes \alpha (y_j)\otimes \alpha (y_i), \\
&&r_{23}r_{13}=\sum _{i, j}\alpha (x_i)\otimes \alpha (x_j)\otimes y_jy_i, \;\;\;\;\;
A(r)=r_{13}r_{12}-r_{12}r_{23}+r_{23}r_{13}. 
\end{eqnarray*}
We say that $r$ 
is a solution of the associative Hom-Yang-Baxter equation if $A(r)=0$, that is  
\begin{eqnarray}
&&\sum _{i, j}\alpha (x_i)\otimes y_ix_j\otimes \alpha (y_j)=\sum _{i, j}x_ix_j\otimes \alpha (y_j)\otimes \alpha (y_i)+
\sum _{i, j}\alpha (x_i)\otimes \alpha (x_j)\otimes y_jy_i.  \label{AHYBE}
\end{eqnarray}
\end{definition}

We introduce the following variation of the concept introduced by Yau in \cite{yauinf}: 
\begin{definition} \label{QTIHB}
An infinitesimal Hom-bialgebra $(A, \mu , \Delta , \alpha )$ is called quasitriangular if there exists an element 
$r\in A\otimes A$, $r=\sum _ix_i\otimes y_i$, such that $(\alpha \otimes \alpha )(r)=r$ and $r$ is a 
solution of the associative Hom-Yang-Baxter equation, with the property that 
\begin{eqnarray*}
&&\Delta (b)=\sum _i\alpha (x_i)\otimes y_ib-\sum _ibx_i\otimes \alpha (y_i), \;\;\;\forall \;\;b\in A. 
\end{eqnarray*}
In this situation, we denote $\Delta $ by $\Delta _r$. 
\end{definition}

Yau's definition requires  
$\Delta (b)=\sum _ibx_i\otimes \alpha (y_i)-\sum _i\alpha (x_i)\otimes y_ib$, for all $b\in A$. 
This is consistent with Aguiar's convention in \cite{aguiarcontemp}; our choice is consistent with 
the convention in \cite{aguiarlectnotes}.  

\begin{definition} (\cite{lmmp1}) A BiHom-dendriform algebra is a 5-tuple $(A, \prec , \succ , \alpha , \beta )$
consisting of a linear space $A$, linear maps $\prec , \succ :A\otimes A\rightarrow A$ and commuting linear maps 
$\alpha , \beta :A\rightarrow A$ such that $\alpha $ and $\beta $ are multiplicative with respect to $\prec $ and $\succ $ and 
satisfying the conditions
\begin{eqnarray}
&&(x\prec y)\prec \beta (z)=\alpha (x)\prec (y\prec z+y\succ z),  \label{BiHomdend6} \\
&&(x \succ y)\prec \beta (z)=\alpha (x)\succ (y\prec z), \label{BiHomdend7} \\
&&\alpha (x)\succ (y\succ z)=(x\prec y+x\succ y)\succ \beta (z), \label{BiHomdend8}
\end{eqnarray}

for all $x, y, z\in A$. We call $\alpha $ and $\beta $ (in this order) the structure maps
of $A$.
\end{definition}

A dendriform algebra, as introduced by Loday in \cite{loday}, is just a BiHom-dendriform algebra 
$(A, \prec , \succ , \alpha , \beta )$ for which $\alpha =\beta=id_A$. A Hom-dendriform algebra, as introduced in 
\cite{makhloufrotabaxter}, is a BiHom-dendriform algebra $(A, \prec , \succ , \alpha , \beta )$ for which $\alpha =\beta $. 

Let $(A, \prec , \succ )$ be a dendriform algebra and $\alpha , \beta :A\rightarrow A$ two
commuting linear maps that are multiplicative with respect to $\prec $ and $\succ $. Define two new operations on $A$ by 
$x\prec _{(\alpha , \beta )}y=\alpha (x)\prec \beta (y)$ and 
$ x\succ _{(\alpha , \beta )}y=\alpha (x)\succ \beta (y)$, 
for all $x, y\in A$. Then $A_{(\alpha , \beta )}:=(A, \prec _{(\alpha , \beta )}, \succ _{(\alpha , \beta )},
\alpha , \beta )$ is a BiHom-dendriform algebra, called the Yau twist of $A$ via the maps $\alpha $ and $\beta $. 
\begin{proposition}(\cite{makhloufrotabaxter}, \cite{makhloufyau}, \cite{lmmp1}) \label{nacer}
Let $(A, \prec , \succ , \alpha , \beta )$ be a BiHom-dendriform algebra and define a new multiplication on $A$ by 
$x*y=x\prec y+x\succ y$. Then $(A, *, \alpha , \beta )$ is a BiHom-associative algebra. Moreover, if $\alpha =\beta $ 
and we define a new operation on $A$ by $x\circ y=x\succ y-y\prec x$, then $(A, \circ , \alpha )$ is a left Hom-pre-Lie algebra. 
\end{proposition}
\section{$\{\sigma, \tau \}$-Rota-Baxter operators}\label{sec2}
\setcounter{equation}{0}
In this section we introduce and study some classes of modified Rota-Baxter operators which are twisted by algebra maps. We recall first the following well-known concept:
\begin{definition}
Let $A$ be an algebra, $\sigma ,\tau :A\rightarrow A$ algebra maps and 
$D:A\rightarrow A$ a linear map. We call $D$ a $(\tau , \sigma )$-derivation if 
$D(ab)=D(a)\tau (b)+\sigma (a)D(b)$, for all $a, b\in A$.
\end{definition}

The following concept is a variation of the one introduced in \cite{panvan} for associative algebras. 
\begin{definition}
Let $A$ be an algebra, $\sigma , \tau :A\rightarrow A$ algebra maps and 
$R:A\rightarrow A$ a linear map. We call $R$ a $(\sigma , \tau )$-Rota-Baxter operator (of weight zero) if 
\begin{eqnarray*}
&&R(a)R(b)=R(\sigma (R(a))b+a\tau (R(b))), \;\;\;\forall \;\;a, b\in A. 
\end{eqnarray*}
\end{definition}
\begin{remark}
For associative algebras, an $(id , \tau )$-Rota-Baxter operator is the same thing as a $\tau $-twisted 
Rota-Baxter operator, a concept introduced in \cite{brzez}. 
\end{remark}
\begin{remark} \label{brzezinski}
Let $R$ be a $(\sigma , \tau )$-Rota-Baxter operator on an associative algebra $A$. One can easily check that 
the triple $(A, \sigma \circ R, \tau \circ R)$ is a Rota-Baxter system, as defined by Brzezi\'{n}ski in \cite{brzez} 
(the case $\sigma =id_A$ may be found in \cite{brzez}). Consequently, by \cite{brzez}, if we define two 
operations on $A$ by $a\prec b=a\tau (R(b))$ and $a\succ b=\sigma (R(a))b$, then 
$(A, \prec , \succ )$ is a dendriform algebra. 
\end{remark}

It is well-know that, if $A$ is an algebra and $D:A\rightarrow A$ is a bijective linear map, 
then $D$ is a derivation (in the usual sense) if and only $D^{-1}$ is a Rota-Baxter operator of weight zero. This fact may be 
easily generalized, as follows: 
\begin{proposition}
Let $A$ be an algebra, $\sigma , \tau :A\rightarrow A$ algebra maps and 
$D:A\rightarrow A$ a bijective linear map with inverse $R=:D^{-1}$. Then $D$ is a $(\tau , \sigma )$-derivation 
if and only if $R$ is a $(\sigma , \tau )$-Rota-Baxter operator.
\end{proposition}

We are interested in the following modification of the concept of $(\sigma , \tau )$-Rota-Baxter operator. 
\begin{definition}
Let $A$ be an algebra, $\sigma , \tau :A\rightarrow A$ algebra maps and 
$R:A\rightarrow A$ a linear map. We call $R$ a $\{\sigma , \tau\}$-Rota-Baxter operator (of weight zero) if 
\begin{eqnarray}
&&R(\sigma (a))R(\tau (b))=R(\sigma (a)R(b)+R(a)\tau (b)), \;\;\;\forall \;\;a, b\in A. \label{genRB}
\end{eqnarray}
\end{definition}
\begin{remark}\label{Rimp}
Let $A$ be an algebra, $\sigma , \tau :A\rightarrow A$ bijective algebra maps and 
$R:A\rightarrow A$ a linear map commuting with $\sigma $ and $\tau $. Then one can easily see that $R$ is a 
$(\sigma , \tau )$-Rota-Baxter operator if and only if $R$ is a $\{\sigma ^{-1}, \tau ^{-1}\}$-Rota-Baxter operator. 
\end{remark}
\begin{remark}
Let $(A, \mu )$ be an algebra, $\sigma :A\rightarrow A$ 
an algebra map and $R:A\rightarrow A$ a Rota-Baxter operator of weight zero commuting with $\sigma $. 
Then one can easily see that $R\circ \sigma$ is a $\{\sigma , \sigma \}$-Rota-Baxter operator both for 
$(A, \mu )$ and for $(A, \sigma \circ \mu )$. 
\end{remark}

We will be particularly interested in the following two classes of $\{\sigma , \tau\}$-Rota-Baxter operators.
\begin{definition}
Let $A$ be an algebra, $\alpha  :A\rightarrow A$ an algebra map, 
$R:A\rightarrow A$ a linear map commuting with $\alpha $ and $n$ a natural number. 
We call $R$ an $\alpha ^n$-Rota-Baxter operator if it is an $\{\alpha ^n , \alpha ^n\}$-Rota-Baxter operator, i.e. 
\begin{eqnarray}
&&R(\alpha ^n(a))R(\alpha ^n(b))=R(\alpha ^n(a)R(b)+R(a)\alpha ^n(b)), \;\;\;\forall \;\;a, b\in A. \label{alfan}
\end{eqnarray}
\end{definition}

Obviously, an $\alpha ^0$-Rota-Baxter operator is just a usual Rota-Baxter operator of weight zero commuting with $\alpha $.
\begin{remark}
From previous remarks it follows that, if $A$ is an algebra, $\alpha :A\rightarrow A$ 
a bijective algebra map, $D:A\rightarrow A$ a bijective linear map commuting with $\alpha $ and $n$ a natural 
number, then $R:=D^{-1}$ is an $\alpha ^n$-Rota-Baxter operator if and only if $D$ is an 
$(\alpha ^{-n}, \alpha ^{-n})$-derivation.  
\end{remark}
\begin{definition}
Let $(A,\mu ,\alpha ,\beta )$ be a BiHom-associative algebra and $R:A\rightarrow A$ a linear map commuting with 
$\alpha $ and $\beta $. 
We call $R$ an $\alpha \beta $-Rota-Baxter operator if it is an $\{\alpha \beta , \alpha \beta \}$-Rota-Baxter operator, 
that is 
\begin{eqnarray}
&&R(\alpha \beta (a))R(\alpha \beta (b))=R(\alpha \beta (a)R(b)+R(a)\alpha \beta (b)), \;\;\;\forall \;\;a, b\in A. \label{generRB}
\end{eqnarray}
\end{definition}
\begin{theorem} \label{simprop}
Let $(A,\mu ,\alpha ,\beta )$ be a BiHom-associative algebra 
and $\sigma , \tau, \eta ,R:A\rightarrow A$ linear maps such that $\sigma , \tau , \eta $ are algebra maps, $R$ is a 
$\{\sigma , \tau\}$-Rota-Baxter operator and any two of the maps $\alpha , \beta , \sigma , \tau, \eta ,R$ commute. 
Define new operations on $A$ by 
\begin{eqnarray*}
&&x\prec y=\sigma (x)R\eta (y) \;\;\;and\;\;\; x\succ y=R(x)\tau \eta (y), 
\end{eqnarray*} 
for all $x, y\in A$. Then $(A, \prec , \succ , \alpha \sigma , \beta \tau \eta )$ is a BiHom-dendriform algebra. 
\end{theorem}
\begin{proof}
One can see that $\alpha \sigma $ and $\beta \tau \eta $ are multiplicative with respect to $\prec $ and $\succ $.  
We compute:
\begin{eqnarray*}
(x\prec y)\prec \beta \tau \eta (z)&=&(\sigma (x)R\eta (y))\prec \beta \tau \eta (z)
=\sigma (\sigma (x)R\eta (y))R\beta \tau \eta ^2(z)\\
&=&(\sigma^2(x)\sigma R\eta (y))\beta R\tau \eta ^2(z)\\
&\overset{(\ref{BHassoc})}{=}&\alpha \sigma^2(x)(R\sigma \eta (y)R\tau \eta ^2(z))\\
&\overset{(\ref{genRB})}{=}&\alpha \sigma^2(x)R(\sigma \eta (y) R\eta ^2(z)+R\eta (y)\tau \eta ^2(z))\\
&=&\alpha \sigma^2(x)R\eta (\sigma (y) R\eta (z)+R(y)\tau \eta (z))\\
&=&\alpha \sigma (x)\prec (\sigma (y) R\eta (z)+R(y)\tau \eta (z))\\
&=&\alpha \sigma (x)\prec (y\prec z+y\succ z). 
\end{eqnarray*}
Then we compute:
\begin{eqnarray*}
(x\succ y)\prec \beta \tau \eta (z)&=&(R(x)\tau \eta (y))\prec \beta \tau \eta (z)
=\sigma (R(x)\tau \eta (y))R\beta \tau \eta ^2(z)\\
&=&(\sigma R(x)\sigma \tau \eta (y))\beta R\tau \eta ^2(z)\\
&\overset{(\ref{BHassoc})}{=}&\alpha \sigma R(x)(\sigma \tau \eta (y)R\tau \eta ^2(z))
=R\alpha \sigma (x) \tau \eta (\sigma (y)R\eta (z))\\
&=&\alpha \sigma (x)\succ (\sigma (y)R\eta (z))=\alpha \sigma (x)\succ (y\prec z). 
\end{eqnarray*}
Also, by using again (\ref{genRB}), one proves that $\alpha \sigma (x)\succ (y\succ z)=(x\prec y+x\succ y)\succ 
\beta \tau \eta (z)$, finishing the proof. 
\end{proof}

We have some particular cases of this theorem. 
\begin{corollary}\label{RBbihom}
Let $A$ be an associative algebra, $\sigma , \tau:A\rightarrow A$ two commuting algebra maps and $R:A\rightarrow A$ 
a $\{\sigma , \tau \}$-Rota-Baxter operator commuting with $\sigma $ and $\tau $. Define new operations on $A$ by 
$x\prec y=\sigma (x)R(y)$ and $x\succ y=R(x)\tau (y)$, for $x, y\in A$. 
Then $(A, \prec , \succ , \sigma , \tau )$ is a BiHom-dendriform algebra. Moreover, if we consider $(A, *, \sigma , \tau )$ 
the BiHom-associative algebra associated to it as in Proposition \ref{nacer}, then $R$ is a 
morphism of BiHom-associative algebras from $(A, * , \sigma , \tau )$ to $A_{(\sigma , \tau )}$, the Yau 
twist of the associative algebra $A$ via the maps $\sigma $ and $\tau $.  
\end{corollary}
\begin{proof}
Take in Theorem \ref{simprop} $\alpha =\beta =\eta =id_A$. The second statement is obvious.  
\end{proof}
\begin{remark}
Assume that we are in the hypotheses of Corollary \ref{RBbihom} and moreover $\sigma $ and $\tau $ are bijective; 
denote $\alpha =\sigma ^{-1}$, $\beta =\tau ^{-1}$. By Remark \ref{Rimp}, $R$ is an $(\alpha , \beta )$-Rota-Baxter operator, 
so, by Remark \ref{brzezinski}, $A$ becomes a dendriform algebra with operations $a\prec b=a\tau ^{-1}(R(b))$ and 
$a\succ b=\sigma ^{-1}(R(a))b$. One can check that the Yau twist of this dendriform algebra via the maps 
$\sigma $ and $\tau $ is exactly the BiHom-dendriform algebra obtained in Corollary \ref{RBbihom}. 
\end{remark}
\begin{corollary}\label{moregendend}
Let $(A, \mu , \alpha )$ be a Hom-associative algebra, $n$ a natural number and $R:A\rightarrow A$ an 
$\alpha ^n$-Rota-Baxter operator. Define new operations on $A$ by 
$x\prec y=\alpha ^n(x)R(y)$ and $x\succ y=R(x)\alpha ^n(y)$, 
for all $x, y\in A$. Then $(A, \prec , \succ , \alpha ^{n+1})$ is a Hom-dendriform algebra. Consequently, by Proposition 
\ref{nacer}, if we define new operations on $A$ by 
\begin{eqnarray*}
&&x*y=x\prec y+x\succ y=\alpha ^n(x)R(y)+R(x)\alpha ^n(y),\\
&&x\circ y=x\succ y-y\prec x=R(x)\alpha ^n(y)-\alpha ^n(y)R(x), 
\end{eqnarray*}
then $(A, *, \alpha ^{n+1})$ is a Hom-associative algebra and $(A, \circ , \alpha ^{n+1})$ is a left Hom-pre-Lie algebra. 
\end{corollary}
\begin{proof}
Take in Theorem \ref{simprop} $\alpha =\beta $, $\sigma =\tau =\alpha ^n$, $\eta =id_A$. 
\end{proof}
\begin{corollary}\label{alphabetaRB}
Let $(A,\mu ,\alpha ,\beta )$ be a BiHom-associative algebra and $R:A\rightarrow A$ an $\alpha \beta $-Rota-Baxter operator. 
Let $\eta :A\rightarrow A$ be an algebra map 
commuting with $\alpha $, $\beta $ and $R$. 
Define new operations on $A$ by 
$x\prec y=\alpha \beta (x)R\eta (y)$  and $x\succ y=R(x)\alpha \beta \eta (y)$, 
for all $x, y\in A$. Then $(A, \prec , \succ , \alpha ^2\beta , \alpha \beta ^2\eta )$ is a BiHom-dendriform algebra. 
\end{corollary}
\begin{proof}
Take in Theorem \ref{simprop} $\sigma =\tau =\alpha \beta $. 
\end{proof}

We recall from \cite{hls} that a Hom-Lie algebra is a triple $\left( L,
\left[\cdot , \cdot \right] ,\alpha \right) $ in which $L$ is a linear
space, $\alpha :L\rightarrow L$ is a linear map and $\left[\cdot , \cdot \right]
:L\times L\rightarrow L$ is a bilinear map,
such that, for all $x, y, z\in L:$
\begin{eqnarray*}
&&\alpha (\left[ x, y \right])=\left[ \alpha \left( x\right),\alpha \left( y \right) \right], \\
&&\left[ x, y \right] =-
\left[ y, x \right],
\;\;\;\; \text{ (skew-symmetry)} \\
&&\left[ \alpha\left( x\right) ,\left[ y, z\right] \right] +\left[ \alpha 
\left( y\right) ,\left[ z, x \right] \right] +\left[ \alpha \left( z\right) ,\left[ x, y \right] \right] =0. 
\text{ (Hom-Jacobi condition)}
\end{eqnarray*}
\begin{proposition}\label{analogLie}
Let $(L, [\cdot , \cdot ], \alpha )$ be a Hom-Lie algebra and $R:L\rightarrow L$ an $\alpha ^n$-Rota-Baxter operator, 
i.e. $R$ commutes with $\alpha $ and 
\begin{eqnarray}
&&[R(\alpha ^n(a)), R(\alpha ^n(b))]=R([\alpha ^n(a), R(b)]+[R(a), \alpha ^n(b)]), \;\;\;\forall \;\;a, b\in L. 
\label{liealfan}
\end{eqnarray}
Then $(L, \cdot , \alpha ^{n+1})$ 
is a left Hom-pre-Lie algebra, where  $a\cdot b=[R(a), \alpha ^n(b)]$, for all $a, b\in L$. 
\end{proposition}
\begin{proof}
Obviously, we have $\alpha ^{n+1}(a\cdot b)=\alpha ^{n+1}(a)\cdot \alpha ^{n+1}(b)$, for all $a, b\in A$. Note that the 
Hom-Jacobi identity together with the skew-symmetry of the bracket $[\cdot , \cdot ]$ imply 
\begin{eqnarray}
&&[\alpha (a), [b, c]]=[[a, b], \alpha (c)]+[\alpha (b), [a, c]], \;\;\;\forall \;\;a, b, c\in A. \label{HL}
\end{eqnarray}
Now for $x, y, z\in A$ we compute: \\[2mm]
${\;\;\;\;\;}$
$\alpha ^{n+1}(x)\cdot (y\cdot z)-(x\cdot y)\cdot \alpha ^{n+1}(z)$
\begin{eqnarray*}
&=&\alpha ^{n+1}(x)\cdot [R(y), \alpha ^n(z)]-[R(x), \alpha ^n(y)]\cdot \alpha ^{n+1}(z)\\
&=&[R(\alpha ^{n+1}(x)), [\alpha ^n(R(y)), \alpha ^{2n}(z)]]-
[R([R(x), \alpha ^n(y)]), \alpha ^{2n+1}(z)]\\
&\overset{(\ref{liealfan})}{=}&[R(\alpha ^{n+1}(x)), [\alpha ^n(R(y)), \alpha ^{2n}(z)]]-
[[R(\alpha ^n(x)), R(\alpha ^n(y))], \alpha ^{2n+1}(z)]\\
&&+[R([\alpha ^n(x), R(y)]), \alpha ^{2n+1}(z)]\\
&=&[R(\alpha ^{n+1}(x)), [R(\alpha ^n(y)), \alpha ^{2n}(z)]]-
[[\alpha ^n(R(x)), R(\alpha ^n(y))], \alpha ^{2n+1}(z)]\\
&&+[R([\alpha ^n(x), R(y)]), \alpha ^{2n+1}(z)]\\
&\overset{(\ref{HL})}{=}&[R(\alpha ^{n+1}(x)), [R(\alpha ^n(y)), \alpha ^{2n}(z)]]-
[\alpha ^{n+1}(R(x)), [R(\alpha ^n(y)), \alpha ^{2n}(z)]]\\
&&+[\alpha ^{n+1}(R(y)), [\alpha ^n(R(x)), \alpha ^{2n}(z)]]+
[R([\alpha ^n(x), R(y)]), \alpha ^{2n+1}(z)]\\
&\overset{skew-symmetry}{=}&[\alpha ^{n+1}(R(y)), [\alpha ^n(R(x)), \alpha ^{2n}(z)]]-
[R([R(y), \alpha ^n(x)]), \alpha ^{2n+1}(z)]\\
&=&\alpha ^{n+1}(y)\cdot (x\cdot z)-(y\cdot x)\cdot \alpha ^{n+1}(z),
\end{eqnarray*}
finishing the proof. 
\end{proof}
\section{The associative BiHom-Yang-Baxter equation}\label{sec3}
\setcounter{equation}{0}
In this section we introduce the associative BiHom-Yang-Baxter equation, generalizing the associative Yang-Baxter equation introduced by Aguiar as well as the associative Hom-Yang-Baxter equation introduced by Yau. Moreover, we discuss its connection with the generalized Rota-Baxter operators introduced in Section \ref{sec2}. 
\begin{definition}
Let $(A,\mu ,\alpha ,\beta )$ be a BiHom-associative algebra and $r=\sum _ix_i\otimes y_i\in A\otimes A$ such that 
$(\alpha \otimes \alpha )(r)=r=(\beta \otimes \beta )(r)$. We define the following elements in $A\otimes A\otimes A$:
\begin{eqnarray*}
&&r_{12}r_{23}=\sum _{i, j}\alpha (x_i)\otimes y_ix_j\otimes \beta (y_j), \;\;\;\;\;
r_{13}r_{12}=\sum _{i, j}x_ix_j\otimes \beta (y_j)\otimes \beta (y_i), \\
&&r_{23}r_{13}=\sum _{i, j}\alpha (x_i)\otimes \alpha (x_j)\otimes y_jy_i, \;\;\;\;\;
A(r)=r_{13}r_{12}-r_{12}r_{23}+r_{23}r_{13}.
\end{eqnarray*}
We say that $r$ is a solution of the associative BiHom-Yang-Baxter equation if $A(r)=0$, i.e. 
\begin{eqnarray}
&&\sum _{i, j}\alpha (x_i)\otimes y_ix_j\otimes \beta (y_j)=\sum _{i, j}x_ix_j\otimes \beta (y_j)\otimes \beta (y_i)+
\sum _{i, j}\alpha (x_i)\otimes \alpha (x_j)\otimes y_jy_i. \label{BHAYBE2}
\end{eqnarray}
\end{definition}
\begin{remark}
Obviously, for $\alpha =\beta $ the associative BiHom-Yang-Baxter equation reduces to the associative 
Hom-Yang-Baxter equation (\ref{AHYBE}). 
\end{remark}
\begin{remark}
Assume that the BiHom-associative algebra $A$ in the previous definition has a unit, that is (see \cite{gmmp}) an 
element $1_A\in A$ satisfying the conditions $\alpha (1_A)=\beta (1_A)=1_A$, $a 1_A=\alpha (a)$ and 
$1_Aa=\beta (a)$, for all $a\in A$. Then, by using the unit $1_A$, one can define the elements 
$r_{12}, r_{13}, r_{23}\in A\otimes A\otimes A$ by $r_{12}=\sum _ix_i\otimes y_i\otimes 1_A$, 
$r_{13}=\sum _ix_i\otimes 1_A\otimes y_i$ and $r_{23}=\sum _i1_A\otimes x_i\otimes y_i$. 
Then, the element $r_{12}r_{23}$ defined above is just the product between $r_{12}$ and $r_{23}$ in 
$A\otimes A\otimes A$, but the element $r_{13}r_{12}$ is \textbf{not} the product between $r_{13}$ and $r_{12}$ 
(which is $\sum _{i, j}x_ix_j\otimes \beta (y_j)\otimes \alpha (y_i)$), and similarly the element 
$r_{23}r_{13}$ is \textbf{not} the product between $r_{23}$ and $r_{13}$. 
\end{remark}
\begin{theorem}\label{ABRB}
Let $(A,\mu ,\alpha ,\beta )$ be a BiHom-associative algebra and $r=\sum _ix_i\otimes y_i\in A\otimes A$ such that $(\alpha \otimes \alpha )(r)=r=(\beta \otimes \beta )(r)$ 
and $r$ is a solution of the associative BiHom-Yang-Baxter equation. Define the linear map 
\begin{eqnarray}
&&R:A\rightarrow A, \;\;\;R(a)=\sum _i\alpha \beta ^3(x_i)(a\alpha ^3(y_i))=
\sum _i(\beta ^3(x_i)a)\alpha ^3\beta (y_i), \;\;\;\;\;\forall \;\;a\in A. \label{RAB}
\end{eqnarray}
Then $R$ is an $\alpha \beta $-Rota-Baxter operator. 
\end{theorem}
\begin{proof}
The fact that $R$ commutes with $\alpha $ and $\beta $ follows immediately from the fact that 
$(\alpha \otimes \alpha )(r)=r=(\beta \otimes \beta )(r)$. Now we compute, for $a, b\in A$: 
\begin{eqnarray*}
R(\alpha \beta (a))R(\alpha \beta (b))
&=&\{\sum _i(\beta ^3(x_i) \alpha \beta (a)) \alpha ^3\beta (y_i)\}
\{\sum _j\alpha \beta ^3(x_j) (\alpha \beta (b) \alpha ^3(y_j))\}\\
&\overset{(\alpha \otimes \alpha )(r)=r}{=}&\{\sum _i(\alpha \beta ^3(x_i) \alpha \beta (a)) \alpha ^4\beta (y_i)\}
\{\sum _j\alpha \beta ^3(x_j) (\alpha \beta (b) \alpha ^3(y_j))\}\\
&\overset{(\ref{BHassoc})}{=}&\sum _{i, j}\{\{(\beta ^3(x_i) \beta (a)) \alpha ^3\beta (y_i)\}
\alpha \beta ^3(x_j)\} \{\alpha \beta ^2(b) \alpha ^3\beta (y_j)\}\\
&\overset{(\ref{BHassoc})}{=}&\sum _{i, j}\{(\alpha \beta ^3(x_i) \alpha \beta (a)) (\alpha ^3\beta (y_i)
\alpha \beta ^2(x_j))\} \{\alpha \beta ^2(b) \alpha ^3\beta (y_j)\}\\
&\overset{(\beta \otimes \beta )(r)=r}{=}&\sum _{i, j}\{(\alpha \beta ^4(x_i) \alpha \beta (a))
(\alpha ^3\beta ^2(y_i)
\alpha \beta ^2(x_j))\} \{\alpha \beta ^2(b) \alpha ^3\beta (y_j)\}\\
&\overset{(\alpha ^2\otimes \alpha ^2)(r)=r}{=}&\sum _{i, j}\{(\alpha \beta ^4(x_i) \alpha \beta (a))
\alpha ^3\beta ^2(y_i
x_j)\} \{\alpha \beta ^2(b) \alpha ^5\beta (y_j)\}\\
&\overset{(\ref{BHAYBE2})}{=}&\sum _{i, j}\{(\beta ^4(x_i x_j) \alpha \beta (a))
\alpha ^3\beta ^3(y_j)\} \{\alpha \beta ^2(b) \alpha ^5\beta (y_i)\}\\
&&+\sum _{i, j}\{(\alpha \beta ^4(x_i) \alpha \beta (a)) 
\alpha ^4\beta ^2(x_j)\} \{\alpha \beta ^2(b) \alpha ^5(y_j y_i)\}\\
&\overset{(\ref{BHassoc})}{=}&\sum _{i, j}\{\alpha \beta ^4(x_i x_j) (\alpha \beta (a) 
\alpha ^3\beta ^2(y_j))\} \{\alpha \beta ^2(b) \alpha ^5\beta (y_i)\}\\
&&+\sum _{i, j}\{(\alpha \beta ^4(x_i) \alpha \beta (a))
\alpha ^4\beta ^2(x_j)\} \{\alpha \beta ^2(b) \alpha ^5(y_j y_i)\}\\
&=&\sum _{i, j}\{(\beta ^3(x_i) \beta ^2(x_j)) (\alpha \beta (a) 
\alpha ^2(y_j))\} \{\alpha \beta ^2(b) \alpha ^4(y_i)\}\\
&&+\sum _{i, j}\{(\beta ^4(x_i) \alpha \beta (a)) 
\beta ^2(x_j)\} \{\alpha \beta ^2(b) (\alpha (y_j) \alpha ^4(y_i))\}, 
\end{eqnarray*}
where for the last equality we used the identities 
$\sum _i \alpha \beta ^4(x_i)\otimes \alpha ^5\beta (y_i)=\sum _i \beta ^3(x_i)\otimes \alpha ^4(y_i)$ and 
$\sum _j \alpha \beta ^4(x_j)\otimes \alpha ^3\beta ^2(y_j)=\sum _j \beta ^2(x_j)\otimes \alpha ^2(y_j)$, 
for the first term, and 
$\sum _i \alpha \beta ^4(x_i)\otimes \alpha ^5(y_i)=\sum _i\beta ^4(x_i)\otimes \alpha ^4(y_i)$ and 
$\sum _j \alpha ^4\beta ^2(x_j)\otimes \alpha ^5(y_j)=\sum _j \beta ^2(x_j)\otimes \alpha (y_j)$, 
for the second term, identities that are consequences of the relation $(\alpha \otimes \alpha )(r)=r=(\beta \otimes \beta )(r)$. 
On the other hand, we have: \\[2mm]
${\;\;\;\;\;}$$R(R(a) \alpha \beta (b)+\alpha \beta (a) R(b))$
\begin{eqnarray*}
&=&R(\sum _j\{\alpha \beta ^3(x_j) (a \alpha ^3(y_j))\} \alpha \beta (b))+
R(\alpha \beta (a) \{\sum _j \alpha \beta ^3(x_j) (b \alpha ^3(y_j))\})\\
&=&\sum _{i, j}\alpha \beta ^3(x_i) \{\{(\alpha \beta ^3(x_j) (a \alpha ^3(y_j))) \alpha \beta (b)\}
 \alpha ^3(y_i)\}\\
&&+\sum _{i, j}\alpha \beta ^3 (x_i) \{\{\alpha \beta (a) (\alpha \beta ^3(x_j) (b \alpha ^3(y_j)))\}
 \alpha ^3(y_i)\}\\
&\overset{(\beta \otimes \beta )(r)=r}{=}&\sum _{i, j}\alpha \beta ^4(x_i) \{\{(\alpha \beta ^3(x_j) (a \alpha ^3(y_j))) \alpha \beta (b)\}\alpha ^3\beta (y_i)\}\\
&&+\sum _{i, j}\alpha \beta ^3 (x_i) \{\{\alpha \beta (a) (\alpha \beta ^3(x_j) (b \alpha ^3(y_j)))\}
 \alpha ^3(y_i)\}\\
&\overset{(\ref{BHassoc})}{=}&\sum _{i, j}\alpha \beta ^4(x_i) \{\{\alpha ^2\beta ^3(x_j) (\alpha (a)
\alpha ^4(y_j))\} \{\alpha \beta (b) \alpha ^3(y_i)\}\}\\
&&+\sum _{i, j}\alpha \beta ^3 (x_i) \{\{(\beta (a) \alpha \beta ^3(x_j)) (\beta (b) \alpha ^3\beta (y_j))\}
 \alpha ^3(y_i)\}\\
&\overset{(\ref{BHassoc})}{=}&\sum _{i, j}\{\beta ^4(x_i) \{\alpha ^2\beta ^3(x_j) (\alpha (a) 
\alpha ^4(y_j))\}\} \{\alpha \beta ^2(b) \alpha ^3\beta (y_i)\}\\
&&+\sum _{i, j}\alpha \beta ^3 (x_i) \{\{(\beta (a) \alpha \beta ^3(x_j)) (\beta (b) \alpha ^3\beta (y_j))\}
 \alpha ^3(y_i)\}\\
&\overset{(\beta \otimes \beta )(r)=r}{=}&\sum _{i, j}\{\beta ^4(x_i)\{\alpha ^2\beta ^3(x_j)(\alpha (a)
\alpha ^4(y_j))\}\}\{\alpha \beta ^2(b) \alpha ^3\beta (y_i)\}\\
&&+\sum _{i, j}\alpha \beta ^4 (x_i) \{\{(\beta (a) \alpha \beta ^3(x_j)) (\beta (b) \alpha ^3\beta (y_j))\}
 \alpha ^3\beta (y_i)\}\\
&\overset{(\ref{BHassoc})}{=}&\sum _{i, j}\{\beta ^4(x_i) \{\alpha ^2\beta ^3(x_j) (\alpha (a) 
\alpha ^4(y_j))\}\} \{\alpha \beta ^2(b) \alpha ^3\beta (y_i)\}\\
&&+\sum _{i, j}\alpha \beta ^4 (x_i) \{(\alpha \beta (a) \alpha ^2\beta ^3(x_j))
\{(\beta (b) \alpha ^3\beta (y_j))
\alpha ^3(y_i)\}\}\\
&\overset{(\alpha \otimes \alpha )(r)=r}{=}&\sum _{i, j}\{\alpha \beta ^4(x_i) \{\alpha ^2\beta ^3(x_j)
 (\alpha (a) 
\alpha ^4(y_j))\}\} \{\alpha \beta ^2(b) \alpha ^4\beta (y_i)\}\\
&&+\sum _{i, j}\alpha \beta ^4 (x_i) \{(\alpha \beta (a) \alpha ^2\beta ^3(x_j)) 
\{(\beta (b) \alpha ^3\beta (y_j))
 \alpha ^3(y_i)\}\}\\
&\overset{(\ref{BHassoc})}{=}&\sum _{i, j}\{(\beta ^4(x_i) \alpha ^2\beta ^3(x_j))
 (\alpha \beta (a)
\alpha ^4\beta (y_j))\} \{\alpha \beta ^2(b) \alpha ^4\beta (y_i)\}\\
&&+\sum _{i, j}\{\beta ^4 (x_i) (\alpha \beta (a) \alpha ^2\beta ^3(x_j))\}
\{(\beta ^2(b) \alpha ^3\beta ^2(y_j))
 \alpha ^3\beta (y_i)\}\\
&\overset{(\alpha \otimes \alpha )(r)=r}{=}&\sum _{i, j}\{(\beta ^4(x_i) \alpha ^2\beta ^3(x_j))
 (\alpha \beta (a)
\alpha ^4\beta (y_j))\} \{\alpha \beta ^2(b) \alpha ^4\beta (y_i)\}\\
&&+\sum _{i, j}\{\alpha \beta ^4 (x_i) (\alpha \beta (a) \alpha ^2\beta ^3(x_j))\}
\{(\beta ^2(b) \alpha ^3\beta ^2(y_j))
 \alpha ^4\beta (y_i)\}\\
&\overset{(\ref{BHassoc})}{=}&\sum _{i, j}\{(\beta ^4(x_i) \alpha ^2\beta ^3(x_j))
 (\alpha \beta (a)
\alpha ^4\beta (y_j))\} \{\alpha \beta ^2(b) \alpha ^4\beta (y_i)\}\\
&&+\sum _{i, j}\{(\beta ^4 (x_i) \alpha \beta (a)) \alpha ^2\beta ^4(x_j)\} 
\{\alpha \beta ^2(b) (\alpha ^3\beta ^2(y_j)
 \alpha ^4(y_i))\}. 
\end{eqnarray*}
By using the identities 
$\sum _i\beta ^4(x_i)\otimes \alpha ^4\beta (y_i)=\sum _i\beta ^3(x_i)\otimes \alpha ^4(y_i)$ and 
$\sum _j\alpha ^2\beta ^3(x_j)\otimes \alpha ^4\beta (y_j)=\sum _j\beta ^2(x_j)\otimes \alpha ^2(y_j)$, 
for the first term, and 
$\sum _j \alpha ^2\beta ^4(x_j)\otimes \alpha ^3\beta ^2(y_j)=\sum _j\beta ^2(x_j)\otimes \alpha (y_j)$,
for the second term, identities that are consequences of the relation $(\alpha \otimes \alpha )(r)=r=
(\beta \otimes \beta )(r)$, the final expression becomes 
\begin{eqnarray*}
&&\sum _{i, j}\{(\beta ^3(x_i) \beta ^2(x_j)) (\alpha \beta (a)
\alpha ^2(y_j))\} \{\alpha \beta ^2(b) \alpha ^4(y_i)\}\\
&&\;\;\;\;\;+\sum _{i, j}\{(\beta ^4(x_i) \alpha \beta (a)) 
\beta ^2(x_j)\} \{\alpha \beta ^2(b) (\alpha (y_j) \alpha ^4(y_i))\}, 
\end{eqnarray*}
and this coincides with the expression we obtained for $R(\alpha \beta (a))R(\alpha \beta (b))$.
\end{proof}
\begin{corollary}\label{RBAHYBE}
Let $(A, \mu , \alpha )$ be a Hom-associative algebra and $r=\sum _ix_i\otimes y_i\in A\otimes A$ such that 
$(\alpha \otimes \alpha )(r)=r$ and $r$ is a solution of the associative Hom-Yang-Baxter equation. 
Define $R:A\rightarrow A$, $R(a)=\sum _i\alpha (x_i)(ay_i)=\sum _i (x_ia)\alpha (y_i)$. Then $R$ is an 
$\alpha ^2$-Rota-Baxter operator. 
\end{corollary}
\begin{proof}
Take $\alpha =\beta $ in the previous theorem and note that, for $\alpha =\beta $, since $(\alpha \otimes \alpha )(r)=r$, 
the formula (\ref{RAB}) becomes $R(a)=\sum _i\alpha (x_i)(ay_i)=\sum _i (x_ia)\alpha (y_i)$.
\end{proof}
\section{Hom-pre-Lie algebras from infinitesimal Hom-bialgebras}\label{sec4}
\setcounter{equation}{0}
In this section we derive Hom-pre-Lie algebras from infinitesimal Hom-bialgebras, generalizing Aguiar's result in the classical case.
\begin{proposition}\label{genGD}
Let $(A, \mu , \alpha )$ be a commutative Hom-associative algebra, $k$ a natural number and $D:A\rightarrow A$ an 
$\alpha ^k$-derivation, that is $D$ is a linear map commuting with $\alpha $ and 
\begin{eqnarray}
&&D(ab)=D(a)\alpha ^k(b)+\alpha ^k(a)D(b), \;\;\;\forall \;\;a, b\in A. \label{gender}
\end{eqnarray}
Define a new operation on $A$ by 
\begin{eqnarray}
&&x\bullet y=\alpha ^k(x)D(y), \;\;\; \forall \;\;x, y\in A. \label{operbullet}
\end{eqnarray}
Then $(A, \bullet , \alpha ^{k+1})$ 
is a Hom-Novikov algebra. 
\end{proposition}
\begin{proof}
Since $D$ commutes with $\alpha $, it is obvious that $\alpha ^{k+1}(x\bullet y)=\alpha ^{k+1}(x)\bullet \alpha ^{k+1}(y)$, 
for all $x, y\in A$. Now we compute:\\[2mm]
${\;\;\;\;\;}$
$\alpha ^{k+1}(x)\bullet (y\bullet z)-(x\bullet y)\bullet \alpha ^{k+1}(z)$
\begin{eqnarray*}
&=&\alpha ^{k+1}(x)\bullet (\alpha ^k(y)D(z))-(\alpha ^k(x)D(y))\bullet \alpha ^{k+1}(z)\\
&=&\alpha ^{2k+1}(x)D(\alpha ^k(y)D(z))-\alpha ^k(\alpha ^k(x)D(y))D(\alpha ^{k+1}(z))\\
&\overset{(\ref{gender})}{=}&\alpha ^{2k+1}(x)(D(\alpha ^k(y))\alpha ^k(D(z))+\alpha ^{2k}(y)D^2(z))\\
&&-
(\alpha ^{2k}(x)\alpha ^k(D(y)))\alpha ^{k+1}(D(z))\\
&\overset{(\ref{Hass})}{=}&\alpha ^{2k+1}(x)(\alpha ^k(D(y))\alpha ^k(D(z)))+
\alpha ^{2k+1}(x)(\alpha ^{2k}(y)D^2(z))\\
&&-\alpha ^{2k+1}(x)(\alpha ^k(D(y))\alpha ^k(D(z)))\\
&\overset{(\ref{Hass})}{=}&(\alpha ^{2k}(x)\alpha ^{2k}(y))\alpha (D^2(z))
=\alpha ^{2k}(xy)\alpha (D^2(z)),
\end{eqnarray*}
and since $xy=yx$, this expression is obviously symmetric in $x$ and $y$, so $(A, \bullet , \alpha ^{k+1})$ is a left  
Hom-pre-Lie algebra.  
Now we compute:
\begin{eqnarray*}
(x\bullet y)\bullet \alpha ^{k+1}(z)&=&(\alpha ^k(x)D(y))\bullet \alpha ^{k+1}(z)
=\alpha ^k(\alpha ^k(x)D(y))D(\alpha ^{k+1}(z))\\
&=&(\alpha ^{2k}(x)\alpha ^k(D(y)))\alpha ^{k+1}(D(z))\\
&\overset{(\ref{Hass})}{=}&\alpha ^{2k+1}(x)(\alpha ^k(D(y))\alpha ^k(D(z)))
=\alpha ^{2k+1}(x)\alpha ^k(D(y)D(z))\\
&\overset{commutativity}{=}&\alpha ^{2k+1}(x)\alpha ^k(D(z)D(y))=(x\bullet z)\bullet \alpha ^{k+1}(y). 
\end{eqnarray*}
So indeed $(A, \bullet , \alpha ^{k+1})$ 
is a Hom-Novikov algebra. 
\end{proof}

By taking $k=0$ in the Proposition, we obtain: 
\begin{corollary}(\cite{yaunovikov}) 
Let $(A, \mu , \alpha )$ be a commutative Hom-associative algebra and $D:A\rightarrow A$ a derivation (in the usual sense)  
commuting with $\alpha $. 
Define a new operation on $A$ by 
$x\bullet y=xD(y)$, for all $x, y\in A$. 
Then $(A, \bullet , \alpha )$ 
is a Hom-Novikov algebra. 
\end{corollary}
\begin{proposition} \label{alphasquare}
Let $(A, \mu , \Delta , \alpha )$ be an infinitesimal Hom-bialgebra. Define the linear map $D:A\rightarrow A$, 
$D(a)=a_1a_2$ for all $a\in A$, i.e. $D=\mu \circ \Delta $. 
Then $D$ is an $\alpha ^2$-derivation. 
\end{proposition}
\begin{proof}
Obviously $D$ commutes with $\alpha $ and, for all $a, b\in A$, we have 
\begin{eqnarray*}
D(ab)&\overset{(\ref{hominf})}{=}&(\alpha (a)b_1)\alpha (b_2)+\alpha (a_1)(a_2\alpha (b))\\
&\overset{(\ref{Hass})}{=}&\alpha ^2(a)(b_1b_2)+(a_1a_2)\alpha ^2(b)=\alpha ^2(a)D(b)+D(a)\alpha ^2(b),
\end{eqnarray*}
finishing the proof.
\end{proof}

Let now $(A, \mu , \Delta , \alpha )$ be a commutative infinitesimal Hom-bialgebra. 
By using Propositions \ref{alphasquare} and \ref{genGD}, we obtain a Hom-Novikov algebra 
$(A, \bullet , \alpha ^3)$, where  
\begin{eqnarray*}
x\bullet y&\overset{(\ref{operbullet})}{=}&\alpha ^2(x)D(y)=\alpha ^2(x)(y_1y_2)\\
&\overset{(\ref{Hass})}{=}&(\alpha (x)y_1)\alpha (y_2)\\
&\overset{commutativity}{=}&(y_1\alpha (x))\alpha (y_2)\\
&\overset{(\ref{Hass})}{=}&\alpha (y_1)(\alpha (x)y_2).
\end{eqnarray*}

Inspired by this, now we have:
\begin{proposition}\label{infprelie}
Let $(A, \mu , \Delta , \alpha )$ be an infinitesimal Hom-bialgebra, and define a new multiplication on $A$ by 
\begin{eqnarray}
&&x\bullet y=\alpha (y_1)(\alpha (x)y_2)=(y_1\alpha (x))\alpha (y_2), \;\;\;\forall \;\;x, y\in A. \label{formbullet}
\end{eqnarray}
Then $(A, \bullet , \alpha ^3)$ is a left Hom-pre-Lie algebra. 
\end{proposition}
\begin{proof}
Since $(\alpha \otimes \alpha )\circ \Delta =\Delta \circ \alpha $, it is easy to see that $\alpha ^3(x\bullet y)=
\alpha ^3(x)\bullet \alpha ^3(y)$, for all $x, y\in A$. Now, for all $x, y, z\in A$ we compute:\\[2mm]
${\;\;}$
$\alpha ^3(x)\bullet (y\bullet z)-(x\bullet y)\bullet \alpha ^3(z)$
\begin{eqnarray*}
&=&\alpha ^3(x)\bullet (\alpha (z_1)(\alpha (y)z_2))-(\alpha (y_1)(\alpha (x)y_2))\bullet \alpha ^3(z)\\
&=&\alpha ([\alpha (z_1)(\alpha (y)z_2)]_1)(\alpha ^4(x)[\alpha (z_1)(\alpha (y)z_2)]_2)\\
&&-\alpha ^4(z_1)\{[\alpha ^2(y_1)(\alpha ^2(x)\alpha (y_2))]\alpha ^3(z_2)\}\\
&\overset{(\ref{hominf})\;twice}{=}&\alpha (\alpha ^2(z_1)(\alpha ^2(y)z_{(2, 1)}))
[\alpha ^4(x)\alpha ^2(z_{(2, 2)})]
+\alpha (\alpha ^2(z_1)\alpha ^2(y_1))[\alpha ^4(x)(\alpha ^2(y_2)\alpha ^2(z_2))]\\
&&+\alpha ^3(z_{(1, 1)})[\alpha ^4(x)(\alpha (z_{(1, 2)})\alpha (\alpha (y)z_2))]
-\alpha ^4(z_1)\{[\alpha ^2(y_1)(\alpha ^2(x)\alpha (y_2))]\alpha ^3(z_2)\}\\
&=&\alpha ([\alpha ^2(z_1)(\alpha ^2(y)z_{(2, 1)})]
\alpha (\alpha ^2(x)z_{(2, 2)}))
+\alpha ^2(\alpha (z_1y_1)[\alpha ^2(x)(y_2z_2)])\\
&&+\alpha (\alpha ^2(z_{(1, 1)})[\alpha ^3(x)(z_{(1, 2)}(\alpha (y)z_2))])
-\alpha (\alpha ^3(z_1)\{[\alpha (y_1)(\alpha (x)y_2)]\alpha ^2(z_2)\}).
\end{eqnarray*}
We claim that the second and fourth terms in this expression cancel each other. To show this, it is enough to prove that 
$\alpha ^2(z_1y_1)[\alpha ^3(x)(\alpha (y_2)\alpha (z_2))]=
\alpha ^3(z_1)\{[\alpha (y_1)(\alpha (x)y_2)]\alpha ^2(z_2)\}$. 
We compute, by applying repeatedly the Hom-associativity condition:
\begin{eqnarray*}
\alpha ^3(z_1)\{[\alpha (y_1)(\alpha (x)y_2)]\alpha ^2(z_2)\}&=&
\alpha ^3(z_1)\{\alpha ^2(y_1)[(\alpha (x)y_2)\alpha (z_2)]\}\\
&=&[\alpha ^2(z_1)\alpha ^2(y_1)][(\alpha ^2(x)\alpha (y_2))\alpha ^2(z_2)]\\
&=&\alpha ^2(z_1y_1)[\alpha ^3(x)(\alpha (y_2)\alpha (z_2))], \;\;\;q.e.d.
\end{eqnarray*}
So, we can now write (by using both Hom-associativity and Hom-coassociativity): \\[2mm]
${\;\;}$
$\alpha ^3(x)\bullet (y\bullet z)-(x\bullet y)\bullet \alpha ^3(z)$
\begin{eqnarray*}
&=&\alpha ([\alpha ^2(z_1)(\alpha ^2(y)z_{(2, 1)})]
\alpha (\alpha ^2(x)z_{(2, 2)}))+\alpha (\alpha ^2(z_{(1, 1)})[\alpha ^3(x)(z_{(1, 2)}(\alpha (y)z_2))])\\
&=&\alpha ([(\alpha (z_1)\alpha ^2(y))\alpha (z_{(2, 1)})]
\alpha (\alpha ^2(x)z_{(2, 2)}))+\alpha ([\alpha (z_{(1, 1)})\alpha ^3(x)]\alpha (z_{(1, 2)}(\alpha (y)z_2)))\\
&=&\alpha ([\alpha ^2(z_1)\alpha ^3(y)][\alpha (z_{(2, 1)})
(\alpha ^2(x)z_{(2, 2)})]+[\alpha (z_{(1, 1)})\alpha ^3(x)][\alpha (z_{(1, 2)})(\alpha ^2(y)\alpha (z_2))])\\
&=&\alpha ([\alpha (z_{(1, 1)})\alpha ^3(y)][\alpha (z_{(1, 2)})
(\alpha ^2(x)\alpha (z_2))]+[\alpha (z_{(1, 1)})\alpha ^3(x)][\alpha (z_{(1, 2)})(\alpha ^2(y)\alpha (z_2))]),
\end{eqnarray*}
and this expression is obviously symmetric in $x$ and $y$. 
\end{proof}
\begin{remark}
The construction introduced in Proposition \ref{infprelie} is compatible with the Yau twist, in the following sense. 
Let $(A, \mu , \Delta )$ be an infinitesimal bialgebra and $\alpha :A\rightarrow A$ a morphism of infinitesimal 
bialgebras. Consider the Yau twist $A_{\alpha }=(A, \mu _{\alpha }=\alpha \circ \mu , 
\Delta _{\alpha }=\Delta \circ \alpha , \alpha )$ (with notation $\mu _{\alpha }(x\otimes y)=x*y=\alpha (xy)$ and 
$\Delta _{\alpha }(x)=x_{[1]}\otimes x_{[2]}=\alpha (x_1)\otimes \alpha (x_2)$), which is an 
infinitesimal Hom-bialgebra, to which we can apply Proposition \ref{infprelie} and obtain a left Hom-pre-Lie algebra 
with structure map $\alpha ^3$ and multiplication 
\begin{eqnarray*}
x\bullet y&=&\alpha (y_{[1]})*(\alpha (x)*y_{[2]})=\alpha ^2(y_1)*\alpha ^2(xy_2)=\alpha ^3(y_1xy_2).
\end{eqnarray*}
This is exactly the Yau twist via the map $\alpha ^3$ of the left pre-Lie algebra obtained from $(A, \mu , \Delta )$ by 
Theorem \ref{zerounu}. 
\end{remark}

Assume now that we have a quasitriangular infinitesimal Hom-bialgebra $(A, \mu , \Delta _r, \alpha )$ as in Definition 
\ref{QTIHB}; 
there are two left Hom-pre-Lie 
algebras that may be associated to $A$, and we want to show that they coincide. 

The first one is $(A, \bullet , \alpha ^3)$ 
obtained from $A$ by using Proposition \ref{infprelie}, with multiplication
\begin{eqnarray*}
a\bullet b&=&\alpha (b_1)(\alpha (a)b_2)=\sum _i\alpha ^2(x_i)(\alpha (a)(y_ib))-\sum _i\alpha (bx_i)(\alpha (a)\alpha (y_i))\\
&\overset{(\ref{Hass})}{=}&\sum _i\alpha ^2(x_i)[(ay_i)\alpha (b)]-\sum _i[\alpha (b)\alpha (x_i)]\alpha (ay_i)\\
&\overset{(\ref{Hass})}{=}&\sum _i[\alpha (x_i)(ay_i)]\alpha ^2(b)-\sum _i\alpha ^2(b)[\alpha (x_i)(ay_i)]. 
\end{eqnarray*}

The second is obtained by applying Corollary \ref{moregendend} (for $n=2$) to the $\alpha ^2$-Rota-Baxter 
operator $R$ defined in Corollary \ref{RBAHYBE}. So, its structure map is $\alpha ^3$ and the 
multiplication is 
\begin{eqnarray*}
&&a\circ b=R(a)\alpha ^2(b)-\alpha ^2(b)R(a)=\sum _i[\alpha (x_i)(ay_i)]\alpha ^2(b)-\sum _i\alpha ^2(b)[\alpha (x_i)(ay_i)], 
\end{eqnarray*}
so indeed $\bullet $ and $\circ $ coincide. 

\subsection*{Acknowledgements}

Ling Liu was supported by NSFC (Grant No.11601486) and  Foundation of Zhejiang Educational Committee (Y201738645).
This paper was written while Claudia Menini was a member of the
''National Group for Algebraic and Geometric Structures, and their Applications'' (GNSAGA-INdAM).
Parts of this paper have been written while
Florin Panaite was a visiting professor at Universit\'{e} de Haute Alsace in May 2017 and a 
visiting professor at 
University of Ferrara in
September 2017 (supported by GNSAGA-INdAM); he would like to thank these institutions for hospitality and support.

\bibliographystyle{amsplain}
\providecommand{\bysame}{\leavevmode\hbox to3em{\hrulefill}\thinspace}
\providecommand{\MR}{\relax\ifhmode\unskip\space\fi MR }
\providecommand{\MRhref}[2]{%
  \href{http://www.ams.org/mathscinet-getitem?mr=#1}{#2}
}
\providecommand{\href}[2]{#2}

\end{document}